\documentclass[reqno]{amsart}
\usepackage{xspace}
\usepackage[bookmarksnumbered,colorlinks]{hyperref}
\usepackage{graphics}

\newcommand{\bt}{\begin{theorem}}                     
\newcommand{\et}{\end{theorem}}                       
\newcommand{\bd}{\begin{definition}}                  
\newcommand{\ed}{\end{definition}}                    
\newcommand{\bl}{\begin{lemma}}                       
\newcommand{\el}{\end{lemma}}                                   
\newcommand{\bpr}{\begin{proposition}}                  
\newcommand{\epr}{\end{proposition}}                    
\newcommand{\bere}{\begin{remark}}                      
\newcommand{\ere}{\end{remark}}                         
\newcommand{\beq}{\begin{equation}}
\newcommand{\eeq}{\end{equation}}
\def\bal#1\eal{\begin{align}#1\end{align}}              
\def\baln#1\ealn{\begin{align*}#1\end{align*}}          
\def\bml#1\eml{\begin{multline}#1\end{multline}}        
\def\bmln#1\emln{\begin{multline*}#1\end{multline*}}  
\def\bga#1\ega{\begin{gather}#1\end{gather}}
\def\bgan#1\egan{\begin{gather*}#1\end{gather*}}
\newcommand{\de}{\mathrm{d}}                        
\newcommand{\N}{\ensuremath{\mathbb{N}}\xspace}     
\newcommand{\R}{\ensuremath{\mathbb{R}}\xspace}     
\newcommand{\eps}{\varepsilon}                      
                   
\newcommand{\inte}{\int_0^1\!\!}

\newtheorem*{theorem*}{Theorem}
\newtheorem{theorem}{Theorem}[section]
\newtheorem{corollary}[theorem]{Corollary}
\newtheorem{lemma}[theorem]{Lemma}
\newtheorem{proposition}[theorem]{Proposition}
\theoremstyle{definition}
\newtheorem{definition}[theorem]{Definition}
\theoremstyle{remark}
\newtheorem{remark}[theorem]{Remark}
\theoremstyle{remark}
\newtheorem{example}[theorem]{Example}

\title[Connecting and closed geodesics of a Kropina metric]%
{Connecting and closed  geodesics\\ of a Kropina metric}

 \author[E. Caponio]{Erasmo Caponio}
 \address{Dipartimento di Meccanica, Matematica  e Management, Politecnico di Bari,  Bari, Italy}
 \email{erasmo.caponio@poliba.it}
 \thanks{E.C. and A. M. are partially supported by  PRIN 2017JPCAPN {\em Qualitative and quantitative aspects of nonlinear PDEs.}}

 \author[F. Giannoni]{Fabio Giannoni}
 \address{Scuola di Scienze e Tecnologie, Universit\`a di Camerino, Camerino, Italy}
 \email{fabio.giannoni@unicam.it}
 
 \author[A. Masiello]{Antonio Masiello}
 \address{Dipartimento di Meccanica, Matematica  e Management, Politecnico di Bari, Bari, Italy}
 \email{antonio.masiello@poliba.it}

 \author[S. Suhr]{Stefan Suhr}
\thanks{Part of this work has been developed while S.S.  was a visiting professor at Politecnico di Bari; he thanks  Politecnico di Bari and the Department of Mechanics, Mathematics and Management for the support. S.S. is partially supported 
by the Deutsche Forschungsgemeinschaft (DFG, German Research Foundation) -- Project-ID 281071066 -- TRR 191.}
\address{Fakult\"at f\"ur Mathematik, Ruhr-Universit\"at Bochum, Bochum, Germany} 
 \email{Stefan.Suhr@rub.de}



\begin{document}

\begin{abstract}
We prove some results about existence of connecting and closed geodesics in a manifold  endowed with a  Kropina metric. These have applications to both null geodesics of spacetimes endowed with a null Killing vector field and Zermelo's navigation problem with critical wind. \end{abstract}
\maketitle
\section{Introduction}\label{intro}
Kropina metrics are  homogeneous Lagrangians defined as the ratio of a Riemannian metric and a one-form, see \cite{Kropina61}.
Let   $S$ be a connected smooth manifold of dimension at least $2$, and let $g_0$, $\omega$  be, respectively, a  Riemannian metric and a one-form on $S$. Let us assume that $\omega$ does not vanish at any point on $S$ and, for each $x\in S$,  let $\mathcal N_x$ be the  kernel  of $\omega_x$ in  $T_xS$. The Kropina metric on $S$ associated to $g_0$ and $\omega$ is the Lagrangian $K\colon TS\setminus \mathcal N\to \R\setminus\{0\}$, defined as $K(v):=\frac{g_0(v,v)}{\omega(v)}$, where $\mathcal N:=\cup_{x\in S}\mathcal N_x$.

For our purposes, it will be convenient to define $K$ on  $\mathcal A=\{v\in TS\colon -\omega(v)>0\}$ as 
\beq\label{kropina}
K(v)=-\frac{g_0(v,v)}{2\omega(v)}
\eeq
in such a way that $K$ is a positive function on $\mathcal A$. On this domain $K$ is a {\em conic Finsler metric} according to  \cite[Definition 3.1 (iii)]{JavSan14}, i.e. at each point $x\in S$, $K_x$ is a Minkowski  norm on $\mathcal A_x$, in particular for all $x\in S$ and for each   $v\in \mathcal A_x$ its 
fundamental tensor
\[\mathbf{g}_v(u,w):= \frac{\partial^2}{\partial t\partial s} \frac{1}{2}K^2(v+tu+sw)_{|t=s=0},\]
$u,w\in T_{x}S$,
is positive definite \cite[Corollary 4.12]{JavSan14}.
We shall often call the couple $(S,K)$ a {\em Kropina space}.

Notice that for all $x\in S$, $0\not\in \mathcal A_x$ though it is an accumulation point of the  {\em indicatrix} $\mathcal I_x$ of $K$ at $x$, i.e. the set of vectors $\{v\in \mathcal A_x: K(v)=1\}$. Thus, $K$ is not extendible by continuity at $0$. We point out that $\{0\}\cup\mathcal I_x$ is a compact strongly convex hypersurface in $T_x S$ (it is an ellipsoid, see \cite[Proposition 2.57]{CaJaSa14}).

The interest in the study of Kropina metrics \eqref{kropina} comes from some  geometric and physical models.

\noindent A first example where a Kropina space $(S,K)$ appears is in general relativity. Let $(S,K)$ be a Kropina space, let us consider the product 
manifold $S\times\R$ and let us denote by $t$ the natural coordinate on $\R$ and by $\pi\colon S\times\R\to S$ the canonical projection.
Let $g$ be the bilinear tensor field on $S\times\R$ defined by
\beq\label{g}
g=\pi^*g_0+\pi^*\omega\otimes\de t+\de t\otimes\pi^*\omega.
\eeq
Since $\omega$ doesn't vanish on $S$, we have that $g$  is a Lorentzian metric,  $t$ is a temporal function and  $-\nabla t$ is timelike (see \cite[Proposition 3.3]{CaJaSa14}). Hence,  $(S\times\R, g)$ is time-oriented by $-\nabla t$; moreover $\partial_t$ is a lightlike Killing vector field.  
Observe now that a vector $(v,\tau)\in T_x S\times \R$ is future pointing and lightlike if and only if $v\in \mathcal A_x$ and $\tau = K(v)$, $K$ in \eqref{kropina}. Analogously,
$(v,\tau)\in T_x S\times \R$ is past pointing and lightlike if and only if $v\in \mathcal - A_x$ and $\tau = -K(-v)$.

Thus, the  future (resp. past) lightlike cones of the conformal class of $(S\times\R, g)$ are described by the flow lines of $\partial_t$ and  the graph of the function $K$ on $\mathcal A$ (resp. $-\mathcal A$).

This correspondence has been introduced in \cite{CaJaSa14} extending that one  between standard stationary Lorentzian metrics and Finsler metrics of Randers type (\cite{CaJaMa11, CaJaS11}). Actually in \cite{CaJaSa14}, the class of spacetimes  $S\times\R$ considered is larger and the Killing vector field $\partial_t$ can also be spacelike in some points (but in this case, the associated  Finsler geometry  is not simply of Kropina or Randers type, see \cite{CaJaSa14} for details).

A second model in which Kropina spaces appear is related to the {\em Zermelo's navigation problem} which consists in finding the paths between two points $x_0$ and $x_1$ that minimize the travel time of a ship or an airship moving in a wind in a Riemannian landscape $(S,g_0)$ (see \cite{Ze31, Carath67,Sh03,YosSab12}). If the wind is time-independent then it can be represented  by a vector field $W$ on $S$. When  $g_0(W,W)=1$,  called {\em critical wind } in \cite{CaJaSa14}, the solutions of the problem (if they exist) are the pregeodesics of the Kropina metric $K(v):=-\frac{g_0(v,v)}{2g_0(W,v)}$ associated to the Zermelo's navigation data $g_0$ and $W$ which are minimizer of the length functional associated to $K$  see \cite[Corollary 6.18 -(i)]{CaJaSa14}. This  result and more general ones contained in \cite{CaJaSa14} are strictly connected to the causality properties of the spacetime $S\times \R$ which is also associated to Zermelo navigation data (\cite[Theorem 6.15]{CaJaSa14}).

Recently,  Kropina metrics have also been considered in relation to  the so-called chains in a CR manifold $M$ \cite{ChMaMM19}. Indeed, these can be viewed as curves which are non-constant projections of  null geodesics for an indefinite metric on a circle bundle over $M$ whose action has infinitesimal generator which is a null Killing vector field. This interpretation leads to a very interesting relation of Kropina spaces with CR geometries and Lorentz geometry, see also 
\cite{Feffe76,Koch88}.

In \cite[Question 2.5.1]{BurMat21}, the authors asked if a Kropina metric on a compact  manifold admits a closed geodesic (this problem was posed in \cite[Remark 6.29]{CaJaSa14} as well).   In  the present  work, we  give some results in that direction  plus results concerning the existence of multiple geodesics between two points.

 The paper is organized as follows. In  Section~\ref{pre}, we introduce  some notations, and we give  some preliminary results; in particular we recall Theorem~\ref{fermatprinciple} from \cite{CaJaSa14} where a connection between geodesics of a Kropina space $(S,K)$ and lightlike geodesics of the spacetime $(S\times\R, g)$, $g$ in \eqref{g} is established.

In Section~\ref{approx}, we introduce an approximation framework of the spacetime $(S\times\R, g)$,  by  a family of spacetimes $(S\times\R, g_{\eps})$, $\eps>0$, where $\partial_t$ is a timelike Killing field (recall that $\partial_t$ is a null Killing field for the spacetime $(S\times\R, g)$). These type of spacetimes are called in the literature {\em standard stationary} and variational methods for their geodesics and their causal properties  are nowdays well-developed 
(cf. \cite{GiaMas91,FoGiMa95, Masiel94, GiaPic99, BaCaFl06, CaFlS08, Masiel09, CaJaMa11, CaJaS11}). 
In particular, Theorem~\ref{fermat} (obtained in \cite{CaJaMa11}) is the precursor of Theorem~\ref{fermatprinciple} and together with Lemma~\ref{deltabounded}, it plays a fundamental role in  proving  existence of geodesics of the   Kropina space associated to the limit  spacetime.
We emphasize that the same approximation  has been  profitably employed in \cite{BaCaFl17} to study geodesic connectedness of a globally hyperbolic spacetime endowed with a null Killing vector field. 
%

In Section~\ref{2points}, we obtain  some results about the existence of geodesics  between two given points of a Kropina space (see Proposition~\ref{geoconnect}, Corollary~\ref{geoconn},  Theorem~\ref{multiplegeoconnect} and Corollary~\ref{cormult}). In particular, Corollary~\ref{cormult} implies that  the Zermelo's navigation problem associated to the data $(S,g_0)$ and $W$, $g_0(W,W)=1$, has always a solution in each connected component of the  space of curves between two points $x_0,\ x_1\in S$ (see  Corollary~\ref{zermelo}).

Section~\ref{closedgeos} is devoted to the closed geodesic problem. Existence results are given in Theorem~\ref{components}, Corollary~\ref{corcomp},  Theorem~\ref{simpconn}.
 In Examples~\ref{katok} and \ref{katokk}, we apply Theorem~\ref{simpconn} to prove the existence of a closed geodesic in some particular type of  compact Kropina space (in particular, Example~\ref{katok}  for an odd-dimensional sphere can be considered as the Kropina limit of a family of Finsler metrics of the type in Katok's example (see \cite{Ziller83}). A couple of results for  a compact manifold endowed with a periodic  Killing vector field $Y$ and a one-form invariant by the flow of $Y$ are further  given in Theorems~\ref{invariant} and \ref{constant}. We notice  that the latter holds for any compact Lie group endowed with  a bi-invariant  Riemannian metric and  a left-invariant one-form (Corollary~\ref{lie}).

Finally,  we point out that  a fundamental and natural  assumption for the existence  of connecting or closed geodesics in a Kropina manifold  is that the space of  paths considered,  according to the boundary conditions that the  geodesics have to  satisfy, must contain at least one admissible path, i.e. a curve $\gamma$ such that $\omega(\dot\gamma(s))<0$  everywhere (under the point of view of the Zermelo's navigation problem, we can  say that there must be a ``navigable region''). The problem of the existence of such an admissible path is  related to the existence of  horizontal paths for the distribution  of hyperplanes pointwise representing the kernel of $\omega$ and, indeed, a non integrability condition for it ensures that there do exist admissible paths  (see Corollaries \ref{cormult} and \ref{corcomp}). Nevertheless, when there exist points in $S$ which are not reachable from a given point by an admissible path, we show in the Appendix that the boundary of the set of reachable points is a smooth hypersurface in $S$.

\section{Some notations and known results}\label{pre}
The set of   continuous and piecewise smooth, admissible curves from $x_0$ to $x_1$ will be denoted by $\Omega_{x_0x_1}(\mathcal A)$, i.e.   
\bmln\Omega_{x_0x_1}(\mathcal A):=\{\gamma\colon[0,1]\to S: \gamma(0)=x_0,\ \gamma(1)=x_1,\\ \dot \gamma^-(s), \dot\gamma^+(s)\in \mathcal A,\ \forall s\in[0,1]\}\emln
(here  $\dot\gamma^-(s)$ and $\dot\gamma^+(s)$ denote respectively the left and the right derivative  of $\gamma$ at the point  $s$).
A geodesic of $(S,K)$  connecting a point $x_0\in S$ to $x_1\in S$ is a critical point of the energy functional 
\beq\label{energy}
E(\gamma)= \frac{1}{2}\int_0^1K^2(\dot \gamma)\de s,
\eeq
defined on $\Omega_{x_0x_1}(\mathcal A)$.
Notice that, as $\mathcal A$ is an open subset of $TS$, variational  vector fields along a curve $\gamma\in \Omega_{x_0x_1}(\mathcal A)$ are well-defined and then it makes sense to define geodesics as critical points of $E$.
Moreover since the fundamental tensor of $K$ is positive definite on $\mathcal A$, it can be proved that the Legendre transform of $K$ is injective (see \cite[Proposition 2.51]{CaJaSa14}) and then  a critical point $\gamma$ of $E$ is smooth and  parametrized with $K(\dot\gamma)=\mathrm{const.}$ (see also \cite[Lemma 2.52]{CaJaSa14}).

Analogously, a closed geodesic is  a critical point of $E$ defined on the set
\[\Omega(\mathcal A)=\{\gamma\colon[0,1]\to S: \gamma(0)=\gamma(1),\ \dot \gamma^-(s), \dot\gamma^+(s)\in \mathcal A,\ \forall s\in[0,1]\},\]

\bere
We observe that, also in the simplest cases, a Kropina space can be not geodesically connected. This essentially may happen because the set $\Omega_{x_0x_1}(\mathcal A)$ is  empty. For example, consider a constant one-form $\omega$ on $\R^n$ endowed with the 
Euclidean metric $\langle\cdot,\cdot\rangle$. It can be easily seen that the geodesic of $\left(\R^n, \frac{\langle\cdot,\cdot\rangle}{\omega(\cdot)}\right)$ are the straight lines which don't lie on the hyperplanes parallel to the kernel of $\omega$. Hence, there is no geodesic (and no admissible curve) connecting two points belonging to one of such hyperplanes. 
\ere
\bere
Notice also that since $K(-v)=-K(v)$, if $\gamma\colon[0,1]\to S$ is a geodesic of $(S,K)$ (hence, according to the above definition, $\dot\gamma([0,1])\subset \mathcal A$) the reverse curve $\tilde \gamma(s)=\gamma(1-s)$ is  a geodesic of $(S, -K)$ with $-K$ viewed as a conic Finsler metric on $-\mathcal A$.
\ere

Geodesics of a Kropina space $(S,K)$ are related to lightlike geodesics of the  product spacetime $(S\times \R,g)$, $g$ as in \eqref{g}. 
We recall that, since the coefficients of the metric $g$ do not depend on the variable $t$, the vector field $\partial_t \equiv (0,1)$ is a  Killing vector field for $(S\times \R,g)$, hence if $z=z(s)=(\gamma(s),t(s))$ is a geodesic of $(S\times \R,g)$ then $g(\partial_t,\dot z)=\omega(\dot x)$ must be constant. Then the following theorem can  be proved 
\begin{theorem}[\cite{CaJaSa14}, Corollary 5.6 (i)]\label{fermatprinciple}
Let $\gamma$ be a piecewise smooth admissible curve in $(S,K)$. Then $\gamma$ is a pregeodesic of the Kropina space $(S,K)$ such that $\omega(\dot\gamma)=\mathrm{const}< 0$
if and only if the curve $z$ defined by
$z(s)=\big(\gamma(s), t(s)= t_0+\int_0^s K(\dot\gamma)\de r\big)$, $t_0\in\R$, is a  future pointing lightlike geodesic of $(S\times \R,g)$ with non-constant component $\gamma$.  
\end{theorem}
Observe that for a given $t_0\in\R$, $z:[a,b]\to S\times\R$ connects the points $(\gamma(a),t_0)$ and  $(\gamma(b), t_0+L(\gamma))$, where $L(\gamma)$ is the Kropina length of $\gamma$, i.e.  
\beq\label{LKrop}
L(\gamma):=\int_a^b K(\dot\gamma)ds.\eeq 
In particular, by \cite[Theorem 5.5 (i)]{CaJaSa14}, if $\gamma\colon [0,T]\to S$ is a non-trivial unit (i.e. $K(\dot\gamma)=\mathrm{const.}=1$) closed geodesic of $(S,K)$ then  $z(t)=(\gamma(t), t)$ is a future pointing lightlike pregeodesic such that its  component $\gamma$ is periodic with period  $T=L(\gamma)$, $z(0)=(\gamma(0), 0)$ and $z(T)=(\gamma(0),T)$.
\bere
When considering lightlike geodesics connecting a point $(x_0, t_0)$ to a flow line of $\partial_t$ passing through a point different from $(x_0, t_0)$, Theorem~\ref{fermatprinciple} can be interpreted as a version, for  spacetimes of the type $(S\times\R,g)$, $g$ as in \eqref{g}, of the Fermat's principle in general relativity, stating that lightlike geodesics connecting a point to a timelike curve $\tau$ are the critical points of the arrival time at the curve $\tau$ (see \cite{Kovner90, Perlic90}). Indeed, in the  class of spacetimes we are considering,  the arrival time of a future pointing lightlike curve $z(s)=\big(\gamma(s),t(s)\big)$ connecting a point $(x_0,t_0)$ to an integral curve of the field $\partial_t$  is $T(\gamma)=t_0+L(\gamma)$. 
We point out that here $\tau$, which is  an integral line of $\partial_t$, is not timelike but lightlike.
\ere

\section{The approximation scheme with standard stationary spacetimes}\label{approx}
In this section we fix a Kropina space $(S,K)$ with a Riemannian metric $g_0$ and a never vanishing one-form $\omega$ on the manifold $S$.
For $\eps>0$, let us consider the standard stationary spacetime $(S\times \R, g_{\eps})$, where
\beq\label{epstationary} g_\eps= \pi^*g_0+\pi^*\omega\otimes\de t+\de t\otimes\pi^*\omega -\eps\de t^2.\eeq
Notice that for each $\eps>0$, $\partial_t$ is a timelike Killing vector field for $g_\eps$.
Let $(x_0,t_0)\in S\times\R$ and $\tau(t)=(x_1,t)$ be the integral line of $\partial_t$ passing through the point $(x_1,0)$.
Let $\nabla$ be the Levi-Civita connection of the Riemannian metric $g_0$. Let us denote by $d_0$ the distance induced on $S$ by $g_0$.  Moreover, let $\|\cdot\|_{x}$ be the norm on the space of linear operators on $T_xS$ endowed with the norm associated to the scalar product $(g_0)_x$, $x\in S$. 

The geodesic equations for $(S\times\R, g_\eps)$ are the following
\beq\label{geotime}
\begin{cases}
	\omega(\dot x_\eps)-\eps \dot t_\eps=c_\eps\\
	\nabla_{\dot x_\eps}\dot x_\eps=\dot t_{\eps}\Omega^\sharp(\dot x_\eps)-\omega^\sharp\ddot t_\eps
\end{cases}
\eeq
while those of $(S\times\R, g)$ are 

\beq\label{geoKrop}\begin{cases}
	\omega(\dot x)=c_0\\
	\nabla_{\dot x}\dot x=\dot t\Omega^\sharp(\dot x)-\omega^\sharp\ddot t
\end{cases}
\eeq
where $c_\eps, c_0\in \R$ and $\omega^\sharp$ and $\Omega^\sharp$ are  the vector field and the $(1,1)$-tensor field $g_0$-metrically equivalent respectively to $\omega$ and $\Omega=\de \omega$. 

The geodesic equations \eqref{geotime} for the standard stationary spacetime $g_\epsilon$ and \eqref{geoKrop} for the metric $g$ can be obtained as the Euler-Lagrange equations of the respective energy functionals
\[
I_\epsilon (x,t) = \frac 1 2\int_0^1g_\epsilon(\dot z,\dot z)\de s = 
\frac 1 2 \int_0^1\left[g_0(\dot x, \dot x) + 2\omega(\dot x)\dot t - \epsilon \dot t^2\right]\de s,
\]
\beq\label{I0}
I_0(x,t) = \frac 1 2 \int_0^1g(\dot z,\dot z)\de s
=\frac 1 2\int_0^1\left[g_0(\dot x, \dot x) + 2\omega(\dot x)\dot t\right]\de s.
\eeq
We point out that the constants $c_\epsilon$ and $c_0$ respectively in \eqref{geotime} and \eqref{geoKrop} derive from the fact that $\partial_t$ is a Killing vector field both for $g_\epsilon$ and $g$, and so it gives rise to the conservation laws $g_\epsilon(\dot z_\epsilon,  \partial_t) = c_\epsilon$ and 
$g(\dot z, \partial_t) = c_0$, where $z_\epsilon = (x_\epsilon,t_\epsilon)$ and $z = (x,t)$ are geodesics respectively for the metric $g_\epsilon$ and $g$.

Let us recall now  the following:
\begin{theorem}[Fermat principle in standard stationary spacetimes \cite{CaJaMa11}]\label{fermat} 
	A curve  $z_\eps\colon[0,1]\to S\times\R$, $z_\eps(s)=(x_\eps(s),t_\eps(s))$ is a future pointing lightlike geodesic of $(S\times\R, g_\eps)$ 
	if and only if $x_\eps$ is a pregeodesic of the Randers  metric on $S$ 
	\[
	F_\eps(v):=\frac{1}{\eps}\left(\sqrt{\eps g_0(v,v)+\omega^2(v)}+\omega(v)\right),\]
	parametrized with $\eps g_0(\dot x_\eps,\dot x_\eps)+\omega^2(\dot x_\eps)=\mathrm{const.}$ and 
	\beq\label{tepss}
	t_\eps(s)=t_0+\int_0^sF_\eps(\dot x_\eps(r))\de r.\eeq
\end{theorem}
From Theorem~\ref{fermat}, for a future pointing lightlike geodesic  of $(S\times\R, g_\eps)$, 
we have 
\beq\label{teps}
t_\eps(1)=t_0+L_\eps(x_\eps),\eeq 
where $L_\eps(x_\eps)$ is the length of $x_\eps$ w.r.t. the Randers metric $F_\eps$, 
\beq\label{Leps}
L_\eps(x_\eps)=\frac{1}{\eps}\int_0^1\left(\sqrt{\eps g_0(\dot x_\eps,\dot x_\eps)+\omega^2(\dot x_\eps)}+\omega(\dot x_\eps)\right)\de r.
\eeq

In the next lemma we give a condition ensuring that  a family  of future pointing lightlike geodesics  $(x_\eps,t_\eps)$ converges uniformly 
to a future pointing lightlike geodesic of $(S\times\R, g)$. Let $H^1([0,1],\R)$ be the Sobolev space of absolute continuous functions on $[0,1]$ with derivative in $L^2$. Moreover let $H^1_0([0,1],\R)$ be the subspace of $H^1([0,1],\R)$ consisting of functions $\tau(s)$ such that $\tau(0) = \tau(1) = 0$.
Henceforth, we denote by $\mathcal P(S)$ both the Sobolev manifolds $\Lambda(S)$ of  $H^1$ free loops on $S$ or $\Lambda_{pq}(S)$ of $H^1$ paths between two points $p$ and $q$ (possibly equal) in $S$.

\bl\label{deltabounded}
Assume that the Riemannian manifold $(S,g_0)$ is complete and there exists a point $\bar x$ and a positive constant $C_{\bar x}$ such that  $\|\omega\|_x\leq C_{\bar x}(d_0(x,\bar x)+1)$.
For each $\eps\in(0,1)$, let $z_\eps=(x_\eps,t_\eps):[0,1]\to S\times \R$ be a future pointing lightlike geodesic of $(S\times\R, g_\eps)$. Let $\Delta_\eps:= t_\eps(1)-t_\eps(0)$ and let us assume that 
$\Delta:= \sup_{\eps\in (0,1)}\Delta_\eps \in\R$.
Then there exists a sequence $\eps_n\to 0$ such that $(x_{\eps_n},t_{\eps_n})$ uniformly converges 
to a curve $z=(x,t)\in\mathcal P(S)\times H^1([0,1],\R)$ which is 
 a future pointing lightlike  geodesic   of $(S\times\R,g)$.
\el
\begin{proof}
	
	By \cite[Theorem 3.3.2 and Eqs. (3.5) and (3.17)]{Masiel94}, for each $\eps>0$, $x_\eps$ is  a critical point of the following functional $J_\eps$ defined on  $\mathcal P(S)$, 
	\[
	J_\eps(x)=\frac{1}{2}\inte g_0(\dot x,\dot x)\de s+\frac{1}{2\eps}\inte \omega^2(\dot x)\de s-\frac{\eps}{2}\left(\Delta_\eps-\frac{1}{\eps}\inte\omega(\dot x)\de s\right)^2\]
	and  $J_\eps(x_\eps)=0$, for all $\eps>0$. 
		Then 
	\[
	\inte g_0(\dot x_\eps,\dot x_\eps)\de s\leq \eps\Delta_\eps^2+2\Delta_\eps\inte|\omega(\dot x_\eps)|\de s
	\leq \Delta^2+2\Delta\inte|\omega(\dot x_\eps)|\de s
	\]
	and, as in  \cite[Lemma 2.6]{BaCaFl06}, we obtain that 
	the family of curves  $(x_\eps)_{\eps\in (0,1)}$ is bounded in $\mathcal P(S)$.
	Moreover, from the second equation in \eqref{geotime}, using the fact that $\Delta_\epsilon$ is bounded, as in Lemma~6.2 in \cite{BaCaFl17} we get that 
	the family  $(\dot t_\eps)_{\eps\in(0,1)}$ is also bounded in $L^2([0,1],\R)$.
	
	Now, for each $\eps$, $z_\eps$ is a critical point of the energy functional $I_\eps$ of the Lorentzian metric $g_\eps$,  i.e.
	\bml\label{critical}\inte \big(g_0(\dot x_\eps, \nabla_{\dot x_\eps}\xi_\eps)+ \dot t_\eps g_0(\nabla_{\xi_\eps}\omega^\sharp, \dot x_\eps)+\dot t_\eps g_0(\omega^\sharp, \nabla_{\dot x_\eps} \xi_\eps)\big )\de s\\ +\inte \omega(\dot x_\eps)\dot \tau\de s
	-\eps \inte\dot t_\eps\dot\tau\de s =0
	\eml
	for any  variational vector field $\xi_\eps$, i.e. $\xi_\eps\in T_{x_\eps}\mathcal P(S)$,  and for any $H^1_0([0,1],\R)$ function $\tau$.
	As $(\dot t_\eps)_{\eps\in(0,1)}$ is bounded in $L^2([0,1],\R)$, we get $\eps\inte \dot t_\eps \dot \tau \to 0$ as $\eps\to 0$.
	Then, as in \cite[Lemma 3.4.1]{Masiel94}, there exists a sequence $\eps_n\to 0$ such that $(x_{\eps_n})$ strongly converges to $x\in \mathcal P(S)$ as $n\to\infty$. Taking in \eqref{critical} $\xi_{\eps_n}=0$, for each $\eps_n$, we get
	$\inte \omega(\dot x_{\eps_n})\dot \tau-\eps_n\inte\dot t_{\eps_n} \dot \tau\de s =0$ and then passing to the limit on $n$, $\inte \omega(\dot x)\dot \tau\de s=0$. Therefore, $\omega(\dot x)$ is constant. Now let $t$ be the weak limit in $H^1([0,1],\R)$ of a sequence $t_{\eps_n}$, $\eps_n\in (0,1)$, $\eps_n\to 0$ as $n\to +\infty$. Then, since $x_{\eps_n}$ strongly converges to $x$,  
	the curve $z(s)=(x(s), t(s))$ satisfies 
	\beq\label{criticalx}
	\inte \big(g_0(\dot x, \nabla_{\dot x}\xi)+ \dot t g_0(\nabla_{\xi}\omega^\sharp, \dot x)+\dot t g_0(\omega^\sharp, \nabla_{\dot x} \xi)\big )\de s=0,
	\eeq
	for any $\xi\in T_x\mathcal P(S)$. Therefore,   $z$ is a critical point of the energy functional 
	$I_0$ in in \eqref{I0}
	of the Lorentzian metric $g$ defined on $\mathcal P(S)\times \left (\{t\}+H^1_0([0,1],\R)\right)$. 
	
	By considering any $H^1$ variational vector field along $z$ with compact support in a neighbourhood $J$ contained in $(0,1)$  of any  instant $s_0\in (0,1)$, and decomposing it in its   components in $x_{|J}^*(TS)$ and $H^1(J, \R)$, this property of $z$ remains true locally. Therefore, as the Lagrangian $(p,v)\in T(S\times\R)\mapsto g_p(v,v)$ is regular,  it can be proved that $z$ is smooth (see, e.g. \cite[p. 609-610]{AbbSch09}) and, therefore, it is a geodesic of $(S\times\R, g)$. (Notice also that if $\mathcal P(S)=\Lambda(S)$, as $x$ satisfies \eqref{criticalx} it must  be a smooth $1$-periodic curve).

	Since $I_{\eps_n}(z_{\eps_n})\to I_0(z)$, as $n \to \infty$, and $I_\eps(z_\eps)=0$, we have that $z$ is lightlike. Finally,  $z$ is future pointing if, by definition, $\dot t(s)>0$ for all $s\in[0,1]$. Notice that  $\dot t(s)$ cannot vanish at some $s\in [0,1]$ because $z$ is lightlike. Hence, $\dot t$ cannot be negative otherwise $0>t(1)-t(0)=\lim_{n\to \infty}\big(t_{\eps_n}(1)-t_{\eps_n}(0)\big)\geq 0$.
\end{proof}
\bere\label{admissible}
Notice that if $\mathcal P(S)=\Lambda(S)$ and  the limit curve $z(s)=(x(s),t(s))$ has component $x$ which is not constant then $-\omega(\dot x)>0$, i.e. $x$ is admissible.  This comes from the fact that, being $z_{\eps_n}$ future pointing in $(S\times\R, g_{\eps_n})$, $0>g_{\eps_n}(\dot z_{\eps_n}, \partial_t)=\omega(\dot x_{\eps_n})-\eps_n \dot t_{\eps_n}=\int_0^1\big(\omega(\dot x_{\eps_n})-\eps_n \dot t_{\eps_n})\de s\to \int_0^1\omega(\dot x)\de s=\omega(\dot x)$, as $n\to\infty$, and the constant $\omega(\dot x)$ cannot be $0$ otherwise, as $z$ is lightlike in $(S\times\R, g)$, we would have   $g_0(\dot x,\dot x)=\mathrm{const.}=0$.   
\ere

\section{The existence of geodesics connecting two points}\label{2points}
We refer to \cite{O'neill} for standard notations and notions  about causality as, e.g.,  the subsets $I^+((x_0,t_0))$ and $J^+((x_0,t_0))$ which represent the set of points in a spacetime $(S\times \R, h)$ which can be connected to $(x_0,t_0)$ by a future-pointing timelike or, respectively,   causal curve.

\bpr\label{geoconnect}
Let $\left(S, -\frac{g_0(\cdot,\cdot)}{2\omega(\cdot)}\right)$ be a Kropina space and $x_0, \ x_1$ be two points on $S$ such that $x_0\neq x_1$ and $\Omega_{x_0x_1}(\mathcal A)\neq \emptyset$. Assume that the Riemannian manifold $(S,g_0)$ is complete and there exists a point $\bar x$ and a positive constant $C_{\bar x}$ such that  $\|\omega\|_x\leq C_{\bar x}(d_0(x,\bar x)+1)$. 
Then there exists a geodesic $\gamma$ of the Kropina space connecting $x_0$ to $x_1$ and  which is a global minimizer of the Kropina length functional on 
$\Omega_{x_0x_1}(\mathcal A)$.
\epr
\begin{proof}
Since $g_0$ is complete and 
$\|\omega\|_x\leq C_{\bar x}(d_0(x_0,x) +1)$, from \cite[Proposition 3.1 and Corollary 3.4]{Sanche97} the spacetimes $(S\times\R, g_\eps)$ are globally hyperbolic, for each $\eps>0$, with Cauchy hypersurfaces $S\times\{t\}$, $t\in\R$. As for any vector $w\in TS\times\R$, $g(w,w)\leq 0$ implies  $g_\eps(w,w)<0$,  also $(S\times\R, g)$ is globally hyperbolic with Cauchy hypersurfaces $S\times\{t\}$, $t\in\R$.

Let $\Upsilon=\{t\in (0,+\infty): (x_1,t_0+t)\in J^+((x_0,t_0))\}$.  Now let $\gamma_0\in \Omega_{x_0x_1}(\mathcal A)$ and consider the curve $z(s)=(\gamma_0(s), t(s))$, with $t(s)=t_0+\int_0^s K(\dot\gamma_0)\de r$, which is lightlike and future pointing in $(S\times\R, g)$. 
Then $L(\gamma_0)\in \Upsilon$, i.e $\Upsilon\neq \emptyset$. Let $T=\inf \Upsilon$. The point $(x_1, t_0+T)$ clearly belongs to $J^+((x_0,t_0))\setminus I^+((x_0,t_0))$ hence there exists a future pointing lightlike geodesic $z(x)=(\gamma(s), t(s))$ connecting $(x_0,t_0)$ to $(x_1, T)$ (see \cite[Proposition 10.46]{O'neill}). By Theorem~\ref{fermatprinciple}, the projection $\gamma$ is a pregeodesic of $(S,K)$ which connects $x_0$ to $x_1$ and minimize the Kropina length functional. 
\end{proof}
\bere
As a globally hyperbolic spacetime is causally simple, Proposition~\ref{geoconnect} can be deduced by  \cite[Theorem 4.9 (i)]{CaJaSa14} which concerns the more general case of a spacetime $S\times \R$ where  $\partial_t$ is a causal Killing vector field.  For a related result see \cite[Proposition 3.22]{JavSan17}. 
\ere

Following \cite[Section 5]{ChMaMM19}, we know that the condition of the existence of an admissible curve between two points in $S$ is ensured provided a non integrability condition for the kernel distribution $\mathcal N$ of $\omega$ is satisfied. Precisely, if $\omega\wedge\de \omega\neq 0 $ in a connected, dense subset of $S$ then there exists a smooth admissible curve between any two points $p$ and $q$ in $S$. Hence we get the following result about geodesic connectedness of a Kropina space that extends \cite[Theorem 1.5]{ChMaMM19} valid in the compact case:
\begin{corollary}\label{geoconn}
Let $(S, g_0)$ be a complete Riemannian manifold	and $\omega$ be a nowhere vanishing one-form such that  there exists a point $\bar x$ and a positive constant $C_{\bar x}$ with   $\|\omega\|_x\leq C_{\bar x}(d_0(x,\bar x)+1)$ and $\omega\wedge\de \omega\neq 0 $ in a connected, dense subset of $S$.
Then the Kropina space $\left(S, -\frac{g_0(\cdot,\cdot)}{2\omega(\cdot)}\right)$ is geodesically connected.
\end{corollary}

A multiplicity result holds if the fundamental group of $S$ is non-trivial. This is based of the observation that  the sequence of the lengths $\Delta_\eps\colon t_\eps(1)-t_\eps(0)$ (or equivalently of the travel times as measured by  observers at infinity in the standard stationary spacetimes  $(S\times\R, g_\eps)$), 
of  the  geodesics $x_\eps$ in Lemma~\ref{deltabounded} can be controlled from above provided these  geodesics  minimize the $F_\eps$-length  in some  homotopy class $\mathcal C\subset \Lambda_{x_0x_1}(S)$ containing at least one admissible curve $\gamma$.

\bt\label{multiplegeoconnect}
Let $\left(S, -\frac{g_0(\cdot,\cdot)}{2\omega(\cdot)}\right)$ be a Kropina space and $x_0, \ x_1$ be two points on $S$, $x_0\neq x_1$. Assume that the Riemannian manifold $(S,g_0)$ is complete and there exists a point $\bar x$ and a positive constant $C_{\bar x}$ such that  $\|\omega\|_x\leq C_{\bar x}(d_0(x,\bar x)+1)$. 
Then for  each connected component $\mathcal C$ of $\Lambda_{x_0x_1}(S)$ there exists a geodesic of the Kropina space from  $x_0$ to $x_1$, which is a  minimizer of the Kropina length functional on $\mathcal C$, provided that there exists an admissible curve $\gamma\in \mathcal C$.  Moreover, if $\mathcal C$ corresponds to a non-trivial element of the fundamental group of $S$, a geodesic loop in $\mathcal C$ exists when  $x_0$ and $x_1$ coincide. 
\et
\begin{proof}
Let  $x_\eps$ be a  geodesic of $(S, F_\epsilon)$ with minimal $F_\eps$-length in  $\mathcal C$. From  \eqref{teps} we get
\[
\Delta_\eps:=L_\eps(x_\eps)\leq L_\eps(\gamma),\]
for each $\eps>0$, where $L_\epsilon$ is defined at \eqref{Leps}.
Since $\gamma\in \mathcal C$ is admissible we obtain
\beq L_\eps(\gamma)=\frac{1}{\eps}\int_0^1-\omega(\dot\gamma)\left(\sqrt{ \frac{\eps g_0(\dot\gamma,\dot\gamma)}{\omega^2(\dot \gamma)}+1}-1\right)\de r \leq \frac{1}{\eps}\int_0^1\frac{1}{2}\frac{\eps g_0(\dot\gamma,\dot\gamma)}{-\omega(\dot \gamma)}\de r=L(\gamma),\label{tbounded}
\eeq
for all $\eps >0$, where $L$ is defined at \eqref{LKrop}. Thus, $\sup_{\eps\in(0,1)}\Delta_\eps\in\R$. After parametrizing $x_\eps$ with $\eps g_0(\dot x_\eps,\dot x_\eps)+\omega^2(\dot x_\eps)=\mathrm{const.}$ (recall  Theorem~\ref{fermat}), a sequence $\eps_n\to 0$ exists such that  the curves $(x_{\eps_n},t_{\eps_n})$  converge to a future pointing lightlike geodesic $(x,t)$ of $(S,g)$ by Lemma~\ref{deltabounded}. Therefore, by  Remark~\ref{admissible}, $x$ is admissible and then by Theorem~\ref{fermatprinciple}, it is a pregeodesic of $(S,K)$. In order to show that it minimizes the length in $\mathcal C$, let us assume that there exists an admissible curve $\gamma_1\in\mathcal C$ such that $L(\gamma_1)<L(x)$. From \eqref{tbounded}, $L_{\eps_n}(x_{\eps_n})\leq L_{\eps_n}(\gamma_1)\leq L(\gamma_1)$; as  $L_{\eps_n}(x_{\eps_n})=t_{\eps_n}(1)-t_{\eps_n}(0)\to t(1)-t(0)=L(x)$ we get a contradiction.
\end{proof}
\begin{corollary}\label{cormult}
Under the assumption of Theorem~\ref{multiplegeoconnect}, assume also that $\omega\wedge\de \omega\neq 0 $ in a connected, dense subset of $S$.
Then for  each connected component $\mathcal C$ of $\Lambda_{x_0x_1}(S)$ there exists a geodesic of the Kropina space from  $x_0$ to $x_1$  which is a  minimizer of the Kropina length functional on $\mathcal C$.  Moreover, if $\mathcal C$ corresponds to a non-trivial element of the fundamental group of $S$, a geodesic loop in $\mathcal C$ exists when  $x_0$ and $x_1$ coincide. 
\end{corollary}	
\begin{proof}
By Theorem~\ref{multiplegeoconnect}, it is enough to show that $\mathcal C$ contains an admissible curve. In fact,  if  $\gamma_0$ is any curve in $\mathcal C$, by \cite[Theorem 1.5 (A)]{ChMaMM19}, we can select a finite number of points $p_j$ belonging to the support of $\gamma_0$ and an equal finite number of convex neighbourhood $U_j$ (convex with respect to the metric $g_0$)  covering $\gamma_0$ such that any point in $U_j$ can be joined to  $p_j$ by an admissible smooth curve (actually a length minimizing geodesic for the Kropina metric). In such a way we obtain a piecewise smooth admissible curve $\gamma$ belonging to the same class $\mathcal C$ of $\gamma_0$.
\end{proof}
As a consequence, from Corollary~\ref{cormult}, under the non integrability assumption for the kernel distribution of the one-form $g_0$-metrically equivalent to $W$, we get that the Zermelo's navigation problem on the complete Riemannian manifold $(S, g_0)$ with critical wind $W$ has a solution in each homotopy class of curves between $x_0$ and $x_1$ in $S$.

\begin{corollary}\label{zermelo}
	Let $(S,g_0)$ be a complete Riemannian metric and $W$ be a vector field on $S$ such that $g_0(W,W)=1$. Then the Zermelo's navigation problem (with data $g_0$ and $W$) between two points $x_0,x_1\in S$, $x_0\neq x_1$, has a solution in each connected component $\mathcal C$ of $\Lambda_{x_0x_1}(S)$ provided that there exists at least one admissible curve in $\mathcal C$. In particular this happens if $\omega\wedge\de \omega\neq 0 $ in a connected, dense subset of $S$. Moreover, if $\mathcal C$ corresponds to a non-trivial element of the fundamental group of $S$, a solution exists  when  $x_0$ and $x_1$ coincide. 	
\end{corollary}
\begin{proof} It is enough to observe that the one-form metrically equivalent to $W$ has constant $g_0$-norm equal to $1$ and therefore the assumption on the growth of $\|\omega\|_x$ in Theorem~\ref{multiplegeoconnect} holds. Thus, the curves which minimize the length functional of the Kropina metric $-\frac{g_0(\cdot,\cdot)}{2g_0(W,\cdot)}$ are  solutions of the Zermelo's navigation problem with data $g_0$ and $W$, see  \cite[Proposition 2.57-(ii) and Corollary 6.18-(i)]{CaJaSa14}.
\end{proof}

\section{The existence of closed geodesics}\label{closedgeos}
We first consider the case when the fundamental group of $S$ is non-trivial. It  is well known that any Finsler metric $F$ on a compact manifold $S$ admits a closed geodesic in each connected component of the free loop space $\Lambda(S)$ which is a minimizer of the energy functional of $F$ and also of its length functional.  As in the statement of Lemma~\ref{deltabounded}, given  a closed geodesic   $x_\eps$ in  $(S,F_\eps)$, we denote by $\Delta_\eps$ the time travel of the corresponding future pointing lightlike pregeodesic  $z_\eps(s)=\big (x_\eps(s),t_\eps(s)\big)$ in $(S\times\R, g_\eps)$, $t_\eps=t_\eps(s)$  given by \eqref{tepss}, which also coincides with the $F_\eps$-length of $x_\eps$.
Therefore, arguing as in the proof of Theorem~\ref{multiplegeoconnect} we obtain the following result that can be interpreted, by the viewpoint of Zermelo's navigation problem, as the possibility of round  trips which minimize the navigation time if the  topology of the sea is non-trivial.
\bt\label{components} Let $S$ be a compact manifold having zero Euler characteristic and non-trivial fundamental group endowed with a Kropina metric $K:=-\frac{g_0(\cdot,\cdot)}{2\omega(\cdot)}$. Then  $(S,K)$ admits  a (non-trivial) closed geodesic with minimal Kropina length in each connected component $\mathcal C$ of $\Lambda(S)$ which does not correspond to  a trivial 
conjugacy class of the fundamental group provided that $\mathcal C$  contains at least one admissible closed curve. 
\et
Analogously to  Corollary~\ref{cormult}, the following also holds:
\begin{corollary}\label{corcomp}
Under the assumptions of Theorem~\ref{components}, assume further  that $\omega\wedge\de \omega\neq 0 $ in a connected, dense subset of $S$. Then $(S,K)$ admits  a (non-trivial) closed geodesic with minimal Kropina length in each connected component $\mathcal C$ of $\Lambda(S)$ which does not correspond to  a trivial 
conjugacy class of the fundamental group. 
\end{corollary}

We give now an existence   result in a  setting  including possibly  the cases that  $S$ is simply connected or non-compact. We also allow $\omega$ vanishing somewhere in $S$, and we denote by  $S_0$  the set of points $x\in S$ where $\omega_x=0$ ($S_0$ being possibly empty).

\bt\label{simpconn}
Let $(S,g_0)$ be a Riemannian manifold and $\omega$ be a one-form on $S$. Let $(\eps_n)_n$ be an infinitesimal sequence of positive numbers and, for each $n\in\N$, $x_{\eps_n}$ be a  closed geodesic of the Randers metric $F_{\eps_n}$. Assume that  $\Delta:=\sup_{n}\Delta_{\eps_n}\in\R$ and the images  of the curves $x_n$ are contained in a compact set  $C$ included in an open subset $U\subset S\setminus S_0$ with compact closure, such that $x_n$ are non-contractible in $U$. 
Then the Kropina space $\left(S\setminus S_0,-\frac{g_0(\cdot,\cdot)}{2\omega(\cdot)}\right)$ admits a (non-trivial) closed geodesic.
\et
\begin{proof}
	We can apply Lemma~\ref{deltabounded} to the sequence of standard stationary spacetimes $\big((S\setminus S_0)\times \R, g_{\eps_n}\big)$. As the images of  $x_{\eps_n}$ are contained in the compact subset $\bar U$ we can assume both  completeness of $g_0$ and boundedness of $\|\omega\|_x$. Hence, up to reparametrization,  the sequence of future-pointing lightlike geodesics $(x_{\eps_n},t_{\eps_n})$ converges uniformly to 
	to a future pointing lightlike geodesic $(x,t)$  of the spacetime  $\big( (S\setminus S_0)\times R,g\big)$ 
	and such that $x$ is a $1$-periodic curve.
	Since $x_n\to x$ in the $C^0$-topology, if $x$ was a constant curve in $\bar C$ then, for $\eps_n$ small enough, $x_{\eps_n}$ would be contractible in $U$.  Therefore, $x$ is a closed pregeodesic of the Kropina space $\left(S\setminus S_0,-\frac{g_0(\cdot,\cdot)}{2\omega(\cdot)}\right)$. 
	\end{proof}
In some cases it is  possible to control from above the lengths $\Delta_\eps$ of the prime closed geodesics of the Randers metrics associated to the approximating stationary spacetimes $(S\times \R, g_\eps)$.
\begin{example}[Kropina limit of Katok metrics]\label{katok}
In order to show a class of examples we need to change a bit the approximation scheme.
Let us replace the one-form $\omega$ in \eqref{epstationary}  by $\sqrt{1-\eps}\omega$,
so that the stationary Lorentzian metric $g_\eps$ is now given by
\[g_\eps:= \pi^*g_0+\sqrt{1-\eps}\pi^*\omega\otimes\de t+\de t\otimes\sqrt{1-\eps}\pi^*\omega -\eps\de t^2.\] 
The Randers metrics 
\[F_\eps(v)=\frac{1}{\eps}\left(\sqrt{\eps g_0(v,v)+(1-\eps)\omega^2(v)}+\sqrt{1-\eps}\omega(v)\right)\] 
defined by the modified $g_\eps$ 
is also associated to the Zermelo data
\[h_\eps=\frac{g_0}{\eps+(1-\eps)\|\omega\|^2}\quad\quad \text{and}\quad\quad  W_\eps=-\sqrt{1-\eps}\omega^\sharp,\] 
see \cite[Proposition 3.6]{CaJaSa14} and \cite[p. 1634]{Robles07}. 
Let  $S=\mathbb S^{n}$, $n\geq 2$ and  $g_0$ be the round metric on $\mathbb S^{n}$.  
Let us assume  that  $\omega^\sharp$ is a Killing vector field for $g_0$. Then the Randers metric $F_\eps$ obtained in this case is one of the celebrated examples  considered by Katok. It is well known that  the resulting Randers  metrics   admit at least $2m$, $n=2m$ or $n=2m-1$, closed geodesics which correspond to  the $m$ great circles $C_i$   invariant by the flow of $\omega^\sharp$, each of them considered twice according to the orientation (see \cite{Ziller83, Robles07}). Let us assume  that  $\|\omega^\sharp\|_x=1$ for all $x\in C_i$ and let us parametrize the $F_\eps$-geodesics   with unit velocity with respect to the metric $g_0$. Then the $F_\eps$ length of these geodesics is given by 
\baln &\Delta_{\eps}=\frac{2\pi}{\eps}(1 -\sqrt{(1-\eps})=\frac{2\pi}{1+\sqrt{1-\eps}}\\
\intertext{if they are parametrized in the same direction of the rotation, and}
&   \Delta_{\eps}=\frac{2\pi}{\eps}(1 +\sqrt{(1-\eps})=\frac{2\pi}{1-\sqrt{1-\eps}}
\ealn
in the other case. Thus, the first family of geodesics has uniformly bounded length and since  for each $i$, their support is $C_i$, they converge to a geodesic of the Kropina space 
$\left(\mathbb S^{2m-1}, -\frac{g_0(\cdot,\cdot)}{2\omega(\cdot)}\right)$ if $n=2m-1$, $\left(\mathbb S^{2m}\setminus\{p,q\}, -\frac{g_0(\cdot,\cdot)}{2\omega(\cdot)}\right)$ if $n=2m$ and $p, q$ are two antipodal points on the sphere where $\omega^\sharp$ vanishes.
Thus, both these Kropina space admits at least $m$ distinct closed geodesics of length $\pi$ and support the $m$ circles $C_i$.
\end{example}
\begin{example}\label{katokk}
The above example can be generalized as follows.	
Let $(S,g_0)$ be a compact Riemannian manifold endowed with a non-trivial periodic Killing vector field $W$ (i.e.  all orbits of $W$ are closed).  It is well known that $(S,g_0)$ admits at least one non-constant closed  geodesic which is one of these orbits (it corresponds to the geodesic with initial conditions $p$ and $W_p$ where  $g_0(W_p,W_p)=\max\{x\in S:g_0(W_x,W_x)\}$). Let $\gamma$ be such a geodesic. Since  $g_0\big (W_{\gamma(t)}, W_{\gamma(t)}\big)$ is constant along $\gamma$ we call it $ c^2$ and we consider the Killing vector field $W/c$. Hence, $g_0(W/c,W/c)\leq 1$.  Let us then consider a perturbation parameter $\alpha<1$. Then the Randers metrics on $S$ defined by the Zermelo navigation data $g_0$ and  $\frac{\sqrt{\alpha}}{c}W$, for each  $\alpha\in(0,1)$,  admit  $\gamma$ as a closed geodesic  by \cite[Theorem 2]{Robles07}, counted twice considering as a different geodesic the one obtained by reversing the orientation of $\gamma$. The shortest of these two geodesics has   length 
\[\Delta_\alpha=\frac{T}{1-\alpha}(1-\sqrt{\alpha})=\frac{T}{1+\sqrt{\alpha}},\]
where $T$ is the $g_0$-length of $\gamma$. As $\Delta_\alpha$  are bounded, from Theorem~\ref{simpconn}  by passing to the limit as $\alpha\to 1$, we conclude that the Kropina metric  $K(v)=\frac{g_0(v,v)}{2g_0(W,v)}$ on  $S\setminus S_0$ admits $\gamma$ as closed geodesic with Kropina length $T/2$. 
\end{example}	
Assuming that the one-form $\omega$ is invariant by the flow of a Killing vector field $Y$ (i.e. the Lie derivative $L_Y\omega$ vanishes) and constant on it, gives a   result about  existence of at least two closed geodesics even when $\omega$ is not the one-form metrically associated to $Y$; in this case the approximation with stationary spacetimes can be bypassed. 
\bt\label{invariant}
Let $(S,g_0)$ be a compact Riemannian manifold, endowed with a non-trivial periodic Killing vector field $Y$, and let $\omega$ be a  one-form such that $\omega(Y)<0$ is constant  and $L_Y \omega=0$. Then there exist at least  two  closed (non-trivial) geodesics of the Kropina metric $K=-\frac{g_0(\cdot,\cdot)}{2\omega(\cdot)}$.
\et
\begin{proof}
Notice that being $\omega(Y)<0$, $Y$ does not vanish at any point of $S$ and then its orbits are non-constant. Then the conclusion follows by observing that there are at least two orbits, passing through a minimizer and a maximizers of the function $p\in S\mapsto g_0(Y_p,Y_p)$ which are geodesics for the Riemannian metric $g_0$ and, under our  assumptions, they are geodesics of the Kropina metric as well. In fact, let $\gamma\colon[0,T] \to S$ be one of these two orbits (the period $T$ depending on $\gamma$).  Since $Y$ is Killing and $L_Y\omega=0$ both $g_0(\dot\gamma,\dot\gamma)$ and $\omega(\dot\gamma)$ are constant along $\gamma$. Moreover, being $\omega(Y)<0$, $\gamma$ is a smooth admissible curve. Then the first variation of the length functional of $K$ with respect to any smooth periodic vector field $\xi$ along $\gamma$ is well-defined and given by
\beq\label{variation}
-\frac{1}{2}\int_0^T\left(\frac{2g_0(\dot\gamma,\nabla^0_\gamma \xi)}{\omega(\dot\gamma)}-\frac{g_0(\dot\gamma,\dot\gamma)\Big(\de \omega(\xi,\dot\gamma)+\frac{\de}{\de s}\big(\omega(\xi)\big)\Big)}{\omega^2(\dot\gamma)}\right)\de s,
\eeq
where $\nabla^0_\gamma$ is the covariant derivative along $\gamma$ induced by the Levi-Civita connection of the metric $g_0$. Hence, integrating by parts and using that $\de\omega(Y, \xi)= (L_Y\omega)(\xi)-\xi\big( \omega(Y)\big)=0$ (recall that $\gamma$ is an orbit of $Y$) we get that the above integral reduces to
\[\int_0^T\frac{g_0(\nabla^0_\gamma\dot\gamma,\xi)}{\omega(\dot\gamma)}\de s\]
which is $0$ for all $\xi$.
\end{proof}	
\bere If $Y$ has constant length too then $\left(S,-\frac{g_0(\cdot,\cdot)}{2\omega(\cdot)}\right)$ admits infinitely many closed geodesics.
\ere
Nevertheless, the assumption $Y$ has constant length can replace $\omega(Y)$ constant.

\bt\label{constant}
Let $(S,g_0)$ be a compact Riemannian manifold, endowed with a non-trivial periodic Killing vector field $Y$ of constant length and let $\omega$ be a 
one-form such that $\omega(Y)<0$ everywhere  and $L_Y \omega=0$. Then there exist at least two closed (non-trivial) geodesics of 
the Kropina metric $K=-\frac{g_0(\cdot,\cdot)}{2\omega(\cdot)}$.
\et
\begin{proof}
As in the previous proof consider the first variation \eqref{variation} of the Kropina length of an orbit $\gamma\colon [0,T]\to S$ of $Y$ in the direction of a periodic vector field
$\xi$ along $\gamma$.
Note that $\omega(\dot\gamma)$ is constant and therefore can be removed from under the integral. Then the first term vanishes since the length of $Y$ 
w.r.t. $g_0$ is constant and therefore $\gamma$ is a $g_0$-geodesic.

Next consider the critical points of the function $f:=p\in S\mapsto \omega_p(Y_p)$. Since $L_Y\omega=0$, $f$  is constant along orbits of
$Y$. Let then $\gamma$ be an orbit whose points are all critical for $f$. With the formula $\de\omega(Y, \xi)= (L_Y\omega)(\xi)-\xi\big( \omega(Y)\big)=0$ 
and the periodicity of $\xi$, we see that the first variation of the Kropina length of $\gamma$ vanishes. As $S$ is compact, there exist at least two critical 
points of $f$ whose values are both negative. The orbits through these two points are the required closed geodesic of $K$.
\end{proof}
Notice that in particular Theorem~\ref{constant} holds with the assumptions of a {\it strong Kropina metric with quasi-regular Killing field} of \cite{SaShYo17}.
 
We finish this section with a result for a Kropina metric on a compact Lie group endowed with a bi-invariant Riemannian metric. 
\begin{corollary}\label{lie}
Let $S$ be a compact Lie group endowed with a bi-invariant Riemannian metric and a non-trivial left-invariant one-form $\omega$. Then the Kropina metric  $K=-\frac{g_0(\cdot,\cdot)}{2\omega(\cdot)}$ on $S$ admits at least two  closed (non-trivial) geodesic. 
\end{corollary}
\begin{proof}
We will construct a right-invariant periodic vector field $Y$ such that $\omega(Y)<0$ everywhere. As $g_0$ is bi-invariant the vector field $Y$ is a Killing 
vector field of $g_0$ with constant length. It further preserves $\omega$ since the form is left-invariant. The claim then follows from Theorem~\ref{constant}.

For the construction of $Y$ we start with an arbitrary right-invariant vector field $X$ on $S$ such that $\omega(X)<0$ everywhere. This is possible since
$\omega$ is non-trivial and left-invariant. Let $H$ be closure of the subgroup generated by $X$. Since $S$ is compact $H$ is a compact abelian group,
i.e. $H\cong T^k$ for some $k$. We can now choose a compact $1$-dimensional subgroup of $H$ whose generator $Y_e$ lies arbitrary close to $X_e$, 
especially $\omega_e(Y_e)<0$. The right-invariant vector field $Y$ associated to $Y_e$ is then the desired right-invariant periodic vector field. 
\end{proof}

%

\begin{appendix}
	\section{On the boundary of a reachable set}
	In light of the problem of existence of paths with finite Kropina length between given points the question begs itself what can be said about the set of 
	reachable points. Recall from the introduction that $\mathcal{A}:=\{v\in TS|-\omega(v)>0\}$, and an absolutely continuous curve $\gamma\colon I\to S$ is 
	 admissible if $\dot{\gamma}(t)\in \mathcal{A}$ for almost all $t\in I$. For $x\in S$, define then $I^\pm_\omega(x)$	to be the sets of terminal or initial points of admissible curves with starting or ending at $x$. It is well known that these sets are open for all $x\in S$ and
	the induced relation is transitive. Under the assumption that $\omega\wedge d\omega \neq 0$ on a dense and connected set, the Chow-Rashevsky 
	Theorem (see \cite[\S 1]{Gromov96}) implies that any pair of points is connected by an admissible curve, i.e. $I^\pm_\omega(x)=S$ for all $x\in S$. 
	
	In analogy to spacetime geometry we will prove the following analog, see \cite[Proposition 14.25]{O'neill}. Note that this is related to the integrability of distributions, see \cite{Suss73, Stefan74}.
	
	\begin{theorem}
		Let $S$ be connected and $x\in S$ such that $I^+_\omega(x)\neq S$. Then the boundary $\Sigma:=\partial I^+_\omega(x)$ is a (non-empty) smooth 
		hypersurface which separates $S$. 
	\end{theorem}
	
	\begin{remark}
		
		\begin{itemize}
			\item[(a)] The opposite set $\partial I^-_\omega(x)$ has the same properties under the appropriate assumptions.
			\item[(b)] By the Chow-Rashevsky Theorem we know that at every point $p\in \partial I^+_\omega(x)$ we have $\omega_p\wedge d\omega_p=0$.
		\end{itemize}
	\end{remark}
	
	\begin{proof}
		The result is local in nature therefore the proof is a local argument.
		
		\underline{$1^{\text{st}}$ step:} $\Sigma$ is a topological hypersurface.
		
		Let $p\in \Sigma$ and let us choose coordinates $\phi=(\phi^1,\ldots,\phi^n)$ in a neighbourhood $U$ of $p$ such that $\phi(p)=0$ and $-\omega_p=(d\phi^n)_p$. By restricting $\phi$ if necessary, we
		can assume that the intersections of $\ker \omega_q$, $q\in U$, with the double cones $\{v\in TU:|d\phi^n(v)|\geq \sum_{i=1}^{n-1} |d\phi^i(v)|\}$ only contain the zero section. This follows from 
		the continuity of $\phi$ and $\omega$. Since any neighbourhood of $p$ contains points in $I^+_\omega(x)$ we obtain that 
		\[
		\mathcal{C}^+:=\left\{q\in U\left|\; \phi^n(q)>\sum_{i=1}^{n-1} |\phi^i(q)|\right.\right\}\subset I^+_\omega(x). 
		\]
		Analogously it follows from the fact that every neighbourhood of $p$ contains points in $\overline{I^+_\omega(x)}^c$ that 
		\[
		\mathcal{C}^-:=\left\{q\in U\left|\; -\phi^n(q)>\sum_{i=1}^{n-1} |\phi^i(q)|\right.\right\}\subset \overline{I^+_\omega(x)}^c. 
		\]
		So every line $t\mapsto (q_1,\ldots,q_{n-1},t)$ will pass $\Sigma$ exactly once. Like in the case of achronal surfaces (see \cite[Proposition 14.25]{O'neill})
		we see that the intersection point depends Lipschitz continuously on $(q_1,\ldots,q_{n-1})$. Therefore, near $p$ the set $\Sigma$ is 
		the graph of a Lipschitz function $\sigma$.  
		
		\underline{$2^{\text{nd}}$ step:} $\Sigma$ is a $C^1$-hypersurface.
		
		We claim that $\Sigma$ has a tangent space everywhere and $T\Sigma_q=\ker\omega_q$, for all $q\in\Sigma$. The claim readily implies the continuous differentiability 
		of $\Sigma$. 
		
		Consider coordinates $\phi$ around $p\in\Sigma$ and function $\sigma$ as before. Further, let $\eta\colon (-\eps,\eps)\to \{\phi^n=0\}$ be a Lipschitz
		curve with $\eta(0)=q$. Then $t\mapsto (\eta(t),\sigma\circ \eta(t))$ is a Lipschitz curve in $\Sigma$. If 
		\[
		\limsup_{t\to 0} \frac{|\sigma\circ \eta(t)|}{|\eta(t)|} 
		\]
		is positive, up to consider larger cones than the ones defined in the first step, we conclude  that $(\eta(t),\sigma\circ \eta(t))\in I^+_\omega(x)$ or $(\eta(t),\sigma\circ \eta(t))\in \overline{I^+_\omega(x)}^c$ for $|t|$ sufficiently small, a contradiction.
		Therefore we have
		\[
		\limsup_{t\to 0} \frac{|\sigma\circ \eta(t)|}{|\eta(t)|} =0
		\]
		for every Lipschitz curve $\eta$, i.e. $T\Sigma_q=\ker (d\phi^n)_q=\ker\omega_q$.
		
		\underline{$3^\text{rd}$ step:} $\Sigma$ is smooth.
		
		Without loosing of generality, we can assume that the codomain of the local coordinates $\phi$ is the cube $[-1,1]^n$. Denote with $\pi_n$ the projection $[-1,1]^n\to [-1,1]^{n-1}$
		which forgets the last coordinate. Under this projection we can uniquely lift the coordinate fields $\partial^\phi_1,\ldots, \partial^\phi_{n-1}$ on 
		$[-1,1]^{n-1}$ to sections $X_1,\ldots,X_{n-1}$ spanning $\ker\omega$ such that $d\pi_n(X_i)=\partial^\phi_i$. Note that the sections are smooth.
		
		Since $\Sigma$ is everywhere tangent to $\ker\omega$ the flow lines of any $X_i$ starting in $\Sigma$ will remain in $\Sigma$. This implies that the flows
		of the $X_i$'s commute along $\Sigma$ and can therefore be used to parameterize the hypersurface $\Sigma$ around $p$ (see for instance \cite[chp. 19]{Lee13}). 
	\end{proof}

\end{appendix}

\end{document}